\documentclass{amsart}
\usepackage{amsmath}
\usepackage{amsfonts}
\usepackage{amssymb}
\usepackage{amsthm}
\usepackage{enumerate}
\usepackage{tikz-cd}
\usepackage{mathrsfs}
\usepackage{wasysym}
\usepackage{hyperref}

\newtheorem{prop}{Proposition}[section]
\newtheorem{cor}[prop]{Corollary}
\newtheorem{lem}[prop]{Lemma}
\newtheorem{thm}[prop]{Theorem}

\newtheorem{defn}[prop]{Definition}
\newtheorem{remark}[prop]{Remark}

\newcommand*{\mb}[1]{\ensuremath{\mathbb{#1}}}
\newcommand*{\mc}[1]{\ensuremath{\mathcal{#1}}}
\newcommand*{\mf}[1]{\ensuremath{\mathfrak{#1}}}

\newcommand*{\Z}{\mb{Z}}
\newcommand*{\Q}{\mb{Q}}

\newcommand*{\dprod}{\displaystyle\prod}

\newcommand*{\clos}[1]{\ensuremath{\overline{#1}}}
\newcommand*{\into}{\hookrightarrow}

\usepackage[OT2,T1]{fontenc}
\DeclareSymbolFont{cyrletters}{OT2}{wncyr}{m}{n}
\DeclareMathSymbol{\Sha}{\mathalpha}{cyrletters}{"58}

\DeclareMathOperator{\Gal}{Gal}

\begin{document}
	
	\title{A Homological Definition for the Tate--Shafarevich Group of a Pell Conic}
	
	\author{Roy Zhao}
	
	\date{\today}
	
	\begin{abstract}
	Franz Lemmermeyer's previous work laid the framework for a description of the arithmetic of Pell conics, which is analogous to that of elliptic curves. He describes a group law on conics and conjectures the existence of an analogous Tate--Shafarevich group with order the squared ideals of the narrow class group. In this article, we provide a cohomological definition of the Tate--Shafarevich group and show that its order is as Lemmermeyer conjectured.
	\end{abstract}
	
	\maketitle
	\tableofcontents
	
	
	\section{Introduction}
A fundamental algebraic invariant associated to an elliptic curve $E/\Q$ is the Tate--Shafarevich group, which measures the failure of the Hasse principle. This group is defined as
\[\Sha(E/\Q) := \ker\left[ H^1(\Q, E(\overline{\Q})) \to \prod_{p \text{ prime}} H^1(\Q_p, E(\overline{\Q_p}))\right].\]
In order to better understand this group, recent work has been done in order to find analogues in simpler cases which we understand better. Lemmermeyer has done exactly this by attempting to find analogies between elliptic curves and conics, and more specifically Pell conics \cite{lemm03}. Pell conics are affine curves given by the equation $x^2 - Dy^2 = 4$, where $D$ is the discriminant of a quadratic number field $K$.

\begin{defn}
	Let $D$ be the discriminant of a quadratic field $K$ and let $\mc{C}: x^2 - Dy^2 = 4$ be a Pell conic. We can endow the curve with a group structure by 
	\[(a, b) + (c, d) := \left(\frac{ac + Dbd}{2}, \frac{ad + bc}{2}\right),\]
	with the identity being $(2, 0)$.
\end{defn}

We look at $x^2 - Dy^2 = 4$ instead of the na\"{i}ve choice $x^2 - Dy^2 = 1$ in order for the $\Z$ points to capture all of the units of norm $1$ in $\mc{O}_K$ if $D \equiv 1 \pmod 4$. Equipped with a group law, we can now define the Tate--Shafarevich group.

\begin{defn}
	We define the Tate--Shafarevich group for a Pell conic $\mc{C}$ as
	\[\Sha(\mc{C}/\Z) := \ker\left(H^1(\Q, \mc{C}(\clos{\Z})) \to \prod_p H^1(\Q_p, \mc{C}(\clos{\Z_p}))\right).\]
\end{defn}

As we will show in Proposition \ref{prop:Cohomology Q points trivial}, we can identify $H^1(\Q, \mc{C}(\clos{\Q})) \cong \Q^\times / N(K^\times)$ and similarly $H^1(\Q_p, \mc{C}(\clos{\Q_p})) \cong \Q_p^\times / N(K_v^\times)$, and consequently using the same definition as the elliptic curve case gives a trivial group by the Hasse Norm Theorem. Hence, we use a definition relating integral points, as opposed to rational points. But, this is not far off from the elliptic curve case because elliptic curves are projective and hence all points can be scaled to be integral. In some sense, looking at integral points as opposed to the rational points captures the difference between an affine and projective curve.

In his paper, Lemmermeyer conjectured the existence of a cohomological definition of a Tate--Shafarevich group with order the squared ideal classes in the narrow class group. In this paper, we prove his conjecture correct and prove the following main theorem.
\begin{thm}
	Using our new definition,
	\[\#\Sha(\mc{C}/\Z) = \# (Cl^+(K)^2).\]
\end{thm}
In Section $2$, we provide some basic algebraic results necessary to prove the theorem in Section $3$. Then, in Section $4$, we set up some machinery that will be used in Section $5$ to provide a geometrical interpretation of the Tate--Shafarevich group. In the elliptic curve case, the group captures principal homogeneous spaces of the curve which are locally trivial. In the case for a Pell conic, we show that the principal homogeneous spaces of a curve correspond exactly to binary quadratic forms, which is where the narrow class group comes from.

As a matter of notation, for a Galois extension $L/K$ and a $\Gal(L/K)$-module $M$, we write $H^i(L/K, M) := H^i(\Gal(L/K), M)$ for the $i$th group cohomology. In the case of the separable closure, we write $H^i(K, M) := H^i(\Gal(K^s/K), M)$. In addition, we denote the class group of a number field $K$ by $Cl(K)$, and the narrow class group by $Cl^+(K)$. We write $\mc{O}_L$ for the ring of integers inside a number field $L$.

	\section{Preliminary Algebraic Results}
First, we prove some basic algebraic results that we will use for the proof of the main theorem.

\begin{prop}\label{ring of integers tensor padic}
	Let $L/K$ be a finite extension of number fields and let $v$ be a non-archimedean place of $K$. Then
	\[\mc{O}_L \otimes_{\mc{O}_K} \mc{O}_{K_v} \cong \prod_{\lambda|v} \mc{O}_{L_\lambda},\]
	where the product is taken over all places $\lambda \mid v$ of $L$.
\end{prop}
\begin{proof}
	Let $\mf{p}$ be the prime associated with $v$. Factoring gives $\mf{p}\mc{O}_L = \dprod_{\lambda \mid v} \mf{q_\lambda}^{e_\lambda}$, where $e_\lambda$ is the ramification index of $\lambda$ over $v$. We know $\mc{O}_{K_v} = \varprojlim_n \mc{O}_K/\mf{p}^n\mc{O}_K$. We can use the Chinese Remainder Theorem to get that
	\[\mc{O}_L \otimes_{\mc{O}_K} \mc{O}_{K_v} = \varprojlim_n \mc{O}_L/\mf{p}^n \mc{O}_L = \varprojlim_n \prod_{\lambda \mid v} \mc{O}_L/\mf{q}_\lambda^{ne_\lambda}\mc{O}_L = \prod_{\lambda \mid v} \varprojlim_n \mc{O}_L/\mf{q}_\lambda^{ne_\lambda}\mc{O}_L = \prod_{\lambda \mid v} \mc{O}_{L_\lambda}.\]
\end{proof}

\begin{lem}\label{lem:Locally equal implies globally equal}
	Let $R$ be a Noetherian ring and $M$ a module. Let $N, N' \subset M$ be finitely generated submodules such that for every maximal ideal $\mf{m} \subset R$, the completions $\hat{N}_\mf{m}, \hat{N}'_\mf{m} \subset \hat{M}_\mf{m}$ are equal. Then $N = N'$.
\end{lem}
\begin{proof}
	Let $N''$ denote the quotient $N'/(N \cap N')$. Since $\hat{R}_\mf{m}$ is flat, we have an exact sequence
	\[0 \rightarrow (N \cap N') \otimes \hat{R}_\mf{m} \rightarrow \hat{N}_\mf{m}' \rightarrow \hat{N}_\mf{m}'' \rightarrow 0.\]
	Again by flatness of $\hat{R}_\mf{m}$, we have $(N \cap N') \otimes \hat{R}_\mf{m} \cong \hat{N}_\mf{m} \cap \hat{N}'_\mf{m} = \hat{N}_\mf{m}'$. Thus $\hat{N}_\mf{m}'' = 0$ and because the inverse system $\{N''/\mf{m}^nN''\}_n$ is surjective, we have $N'' = \mf{m}N''$ for all $\mf{m}$. An application of Nakayama's lemma gives us $N'' = 0$ and $N \cap N' = N'$, so $N' \subset N$. By symmetry, the other direction holds and therefore $N = N'$.
\end{proof}

The following is the main lemma that will be used in the proof. It allows us to treat the local condition that a cocyle is trivial locally as a global condition that the cocycle becomes trivial under a base change.
\begin{lem}\label{lem:group scheme local kernel}
	Let $L/K$ be a finite Galois extension of number fields. Let $v$ be a non-archimedean place of $K$ and $\lambda$ a place of $L$ lying over $v$. Let $G/\mc{O}_K$ be a group scheme and suppose that for all $\mc{O}_K$-algebras $A, B$, the $A$ and $B$ points satisfy $G(A \times B) \cong G(A) \times G(B)$. Then
	\begin{align*}&\ker\left[H^1(L/K, G(\mc{O}_L)) \to H^1(L_\lambda/ K_v, G(\mc{O}_{L_\lambda}))\right]\\
	= &\ker\left[H^1(L/K, G(\mc{O}_L)) \to H^1(L/K, G(\mc{O}_L \otimes_{\mc{O}_K} \mc{O}_{K_v}))\right].
	\end{align*}
\end{lem}
\begin{proof}
	Corollary \ref{ring of integers tensor padic} says that $\mc{O}_L \otimes \mc{O}_{K_v} \cong \dprod_{\lambda' \mid v} \mc{O}_{L_{\lambda'}}$. The projection onto $\mc{O}_{L_\lambda}$ gives a map $\mc{O}_L \otimes \mc{O}_{K_v} \to \mc{O}_{L_\lambda}$ that makes the following diagram commute.
	\[\begin{tikzcd} H^1(L/K, G(\mc{O}_L)) \arrow[r] \arrow[dr] & H^1(L/K, G(\mc{O}_L \otimes \mc{O}_{K_v})) \arrow[d] \\ & H^1(L_\lambda/K_v, G(\mc{O}_{L_\lambda}))\end{tikzcd}\]
	Thus, the kernel of the top morphism is automatically in the kernel of the diagonal morphism.
	
	Now take a cocycle $c$ in the kernel of the diagonal map. The locally trivial condition on $c$ tells us that there exists some $m_\lambda \in G(\mc{O}_{L_\lambda})$ such that for all $\sigma \in \Gal(L_\lambda/K_v)$, we have $c(\sigma) = \sigma(m_\lambda) - m_\lambda$. Fix such an $m_\lambda$. Our goal is to generate an element $m_v \in G(\mc{O}_L \otimes \mc{O}_{K_v})$ such that $c(\sigma) = \sigma(m_v) - m_v$ for all $\sigma \in \Gal(L/K)$. First let $\mf{p} \subset \mc{O}_K$ be the prime ideal associated with $v$ and factorize it as $\mf{p}\mc{O}_L = \dprod_{\lambda'} \mf{q}_{\lambda'}^e$. Let $\{\tau_{\lambda'}\}_{\lambda'}$ denote a set of representatives of the cosets of $\Gal(L/K)/\Gal(L_\lambda/K_v)$, where $\tau_{\lambda'}(\mf{q}_\lambda) = \mf{q}_{\lambda'}$. For each place $\lambda'$, the automorphism $\tau_{\lambda'}$ induces an isomorphism $\mc{O}_{L_\lambda} \to \mc{O}_{L_{\lambda'}}$. We claim that since $c$ is a boundary when we embed $\mc{O}_L$ into $\mc{O}_{L_\lambda}$, it is also a boundary regardless of the choice of place $\lambda'$. To see why, for any $\sigma \in \Gal(L_{\lambda'}/K_v)$, there exists a unique $\sigma' \in \Gal(L_\lambda/K_v)$ such that $\sigma = \tau_{\lambda'} \sigma' \tau_{\lambda'}^{-1}$ and using the $1$-cocycle condition repeatedly gives
	\[c(\sigma) = c(\tau_{\lambda'}\sigma' \tau_{\lambda'}) = \sigma(\tau_{\lambda'}(m_\lambda) - c(\tau_{\lambda'})) - (\tau_{\lambda'}(m_\lambda) - c(\tau_{\lambda'})),\]
	thus proving our claim. This is independent of the choice of representative because for another representative $\tau_{\lambda'}'$, there must exist $\sigma \in \Gal(L_\lambda/K_v)$ such that $\tau_{\lambda'} = \tau_{\lambda'}'\sigma$. Then
	\[\tau_{\lambda'}(m_\lambda) - c(\tau_{\lambda'}) = \tau'_{\lambda'}(\sigma(m_\lambda)) - c(\tau_{\lambda'}'\sigma) = \tau'_{\lambda'}(\sigma(m_\lambda)) - (\tau_{\lambda'}'(c(\sigma)) + c(\tau_{\lambda'}'))\]
	\[= \tau'_{\lambda'} (\sigma(m_\lambda) - (\sigma(m_\lambda) - m_\lambda)) - c(\tau'_{\lambda'}) = \tau'_{\lambda'}(m_\lambda) - c(\tau_{\lambda'}').\]
	We know $\mc{O}_L \otimes \mc{O}_{K_v}\cong \dprod_\lambda \mc{O}_{L_\lambda}$ and since $G(A \times B) \cong G(A) \times G(B)$, specifying $m_v \in G(\mc{O}_L \otimes \mc{O}_{K_v}) \cong \dprod_\lambda G(\mc{O}_{L_\lambda})$ is equivalent to specifying an element in $G(\mc{O}_{L_{\lambda'}})$ for each place $\lambda'$. Let $m_v = \dprod_{\lambda'} [\tau_{\lambda'}(m_\lambda) - c(\tau_{\lambda'})]$. Now we need to show that $m_v$ satisfies $c(\sigma) = \sigma(m_v) - m_v$. Since we chose $\lambda$ arbitrarily and since we showed that $m_v$ is independent of our choice of representatives, it suffices to show that if $\sigma(\mf{p}_\lambda) = \mf{p}_{\lambda'}$, then the $\lambda'$th component of $\sigma(m_v) - m_v$ is $c(\sigma)$. Let $\sigma = \tau_{\lambda'}\sigma'$ for $\sigma' \in \Gal(L_\lambda/K_v)$. The $\lambda'$th component of $\sigma(m_v) - m_v$ is
	\[\tau_{\lambda'}(\sigma'(m_\lambda)) - (\tau_{\lambda'}(m_\lambda) - c(\tau_{\lambda'})) = \tau_{\lambda'}(c(\sigma')) + c(\tau_{\lambda'}) = c(\tau_{\lambda'} \sigma') = c(\sigma).\]
	Thus, we have found an $m_v \in G(\mc{O}_L \otimes \mc{O}_{K_v})$ which finishes showing the equality of the two kernels.
\end{proof}

\begin{lem}\label{lem:quadratic form isomorphism}
	Let $K$ be a quadratic number field of discriminant $D$ and $\mf{a} = ca \Z+ c\frac{b + \sqrt{D}}{2}\Z$ a fractional ideal of $K$. Let $A$ be a $\Z$-algebra and
	\[q_\mf{a}(x, y) = \frac{N(ca x + c\frac{b+ \sqrt{D}}{2} y)}{N(\mf{a})} = ax^2 + bxy + \frac{b^2 - D}{4a} y^2,\]
	the binary quadratic form associated with $\mf{a}$. Then $\mf{a} \otimes_\Z A \subset K \otimes A$ is a cyclic $\mc{O}_K \otimes A$ module if and only if there exist $x, y \in A$ such that $q_\mf{a}(x, y) \in A^\times$.
\end{lem}
\begin{proof}
	Let $\alpha = ca, \beta = c\frac{b + \sqrt{D}}{2}$ be generators for $\mf{a}$. Then we can write an element $m \in \mf{a} \otimes A$ as $m = \alpha \otimes x + \beta \otimes y$. For the equality $\mf{a} \otimes A = m(\mc{O}_K \otimes A)$ to hold, there must exist some $d, d' \in \mc{O}_K \otimes_\Z A$ such that $m d = \alpha \otimes 1$ and $m d' = \beta \otimes 1$. An arbitrary $d$ is given by $1 \otimes x' + \delta \otimes y'$ and multiplying gives
	\[m d = \alpha \otimes \left(x x' + \frac{D - b}{2}xy' + \frac{D - b^2}{4a} yy'\right) + \beta \otimes \left(yx' + axy' + \frac{D+ b}{2}yy'\right).\]
	Therefore, such an element $m$ exists if and only the matrix
	\[\begin{pmatrix}
	x & \frac{D - b}{2}x + \frac{D - b^2}{4a}y\\ y & ax + \frac{D+ b}{2}y
	\end{pmatrix}\]
	is invertible in $M_2(A)$, which occurs if and only if its determinant is invertible. So there exists an $m$ satisfying $\mf{a} \otimes A = m(\mc{O}_K \otimes A)$ if and only if there exist $x, y \in A$ such that
	\[q_\mf{a}(x, y) = ax^2 + bxy + \frac{b^2 - D}{4a}y^2 = \left|\begin{matrix}
	x & \frac{D - b}{2}x + \frac{D - b^2}{4a}y\\ y & ax + \frac{D+ b}{2}y
	\end{matrix}\right| \in A^\times.\]
\end{proof}

\begin{thm}
	Let $K$ be a number field and $f(X, Y) = \sum_{i = 0}^d h_i X^{d - i}Y^i$ be a homogeneous form of degree $d$ with coefficients in $\mc{O}_K$. Suppose that $k \in (h_0, \dots, h_d) \subset \mc{O}_K$. Then, there is a finite algebraic extension $L/K$ and integers $x, y \in \mc{O}_L$ such that $f(x, y) = k$.
\end{thm}
\begin{proof}
	A proof is provided by Birch \cite{birch1985}.
\end{proof}

\begin{cor}\label{cor:binary quadratic form 1}
	Let $\mf{a}$ be a fractional ideal of quadratic number field $K$ and let $q_\mf{a}$ be the associated binary quadratic form. Then there exists a finite extension $L/\Q$ and integers $x, y \in \mc{O}_L$ such that $q_\mf{a}(x, y) = 1$.
\end{cor}
\begin{proof}
	The binary quadratic form $q_\mf{a}$ is primitive and has coprime coefficients.
\end{proof}

\section{An Argument Using Twisted Forms}
The general idea of the proof presented will be to associate each cohomology class with a twisted Galois action of $\Gal(L/\Q)$ on $\mc{C}(\mc{O}_L)$, and then consider the points fixed by this new action. The set of fixed points will be a $\mc{O}_K$-module, which we will show to be isomorphic to fractional ideals of $K$. Thus, our problem will be reduced to classifying which ideals come from cohomology classes. First though, we prove that the kernel is trivial when considering rational points, as this fact will be used in the proof of the main theorem.

\begin{prop}\label{prop:Cohomology Q points trivial}
	\[\ker\left(H^1(\Q, \mc{C}(\clos{\Q})) \to \prod_p H^1(\Q_p, \mc{C}(\clos{\Q_p}))\right) = 1.\]
\end{prop}
\begin{proof}
	Note that there exists an isomorphism of $\Gal(\clos{K}/K)$ (resp. $\Gal(\clos{K_v}/K_v)$) modules between $\clos{K}^\times$ and $\mc{C}(\clos{K})$ (resp. $\clos{K_v}^\times$ and $\mc{C}(\clos{K_v})$) by $c \mapsto (c + c^{-1}, \frac{c - c^{-1}}{\sqrt{D}})$ and an inverse given by $(a, b) \mapsto \frac{a + b\sqrt{D}}{2}$. The inflation restriction sequence gives us an exact sequence
	\[\begin{tikzcd}
	1 \arrow[r] & H^1(K/\Q, \mc{C}(\clos{\Q})^{\Gal(\clos{K}/K)}) \arrow[r] & H^1(\Q, \mc{C}(\clos{\Q})) \arrow[r] & H^1(K, \mc{C}(\clos{K})).
	\end{tikzcd}.\]
	Hilbert's Theorem $90$ tells us that $H^1(K, \mc{C}(\clos{K})) \cong H^1(K, \clos{K}^\times) = 1$. We know $\mc{C}(\clos{Q})^{\Gal(\clos{K}/K)} = \mc{C}(K)$, and so we have $H^1(\Q, \mc{C}(\clos{\Q})) \cong H^1(K/\Q, \mc{C}(K))$ (resp. $H^1(\Q_p, \mc{C}(\clos{\Q_p})) \cong H^1(K_v/\Q_p, \mc{C}(K_v))$).
	
	The $1$-cocycles of $H^1(K/\Q, \mc{C}(K))$ are determined by the image of $\sigma: \sqrt{D} \mapsto -\sqrt{D}$. This image must satisfy $\sigma(x, y) + (x, y) = (2, 0)$ and identifying each point $(x, y) \in \mc{C}(K)$ with a unique element $c \in K^\times$, we have that $\sigma(x, y)$ is identified with $\frac{\sigma(x) + \sigma(y)\sqrt{D}}{2} = \sigma(\frac{x - y\sqrt{D}}{2})$. So cocycles are identified with $c \in K^\times$ satisfying $\sigma(c^{-1})c = 1$, meaning $c \in \Q^\times$. Under the same identification, the $1$-coboundaries become identified with elements of the form $\sigma(x, y) - (x, y) = \sigma(c^{-1})\cdot c^{-1} = N(c^{-1})$. In this way $H^1(K/\Q, \mc{C}(K)) \cong \Q^\times/N(K^\times)$, where $N(K^\times)$ is the set of elements of $\Q^\times$ that arise as norms from elements in $K^\times$. Thus,
	\[\ker\left(H^1(\Q, \mc{C}(\bar{\Q})) \to \prod_p H^1(\Q_p, \mc{C}(\bar{\Q_p}))\right) = \ker\left(\Q^\times/N(K^\times) \to \prod_p \Q_p^\times/N(K_v^\times)\right),\]
	and the latter is trivial by the Hasse Norm Theorem.
\end{proof}

\begin{defn}
	Let $K$ be a quadratic field of discriminant $D$. Let $N_1$ be an algebraic curve whose $A$ points are
	\[N_1(A) := (\mc{O}_K \otimes A)^\times_1 := \{x \in (\mc{O}_K \otimes A)^\times: N(x) = 1\},\]
	where the norm is induced form the norm $\mc{O}_K \to \Z$ and takes $\mc{O}_K \otimes A \to A$. This gives $N_1$ the structure of an algebraic group. To be explicit, we have $\mc{O}_K = \Z[\delta]$ for $\delta := \frac{D+ \sqrt{D}}{2}$ and hence $N_1: x^2 + Dxy + \frac{D^2 - D}{4}y^2 = 1$. We will sometimes write an $A$ point as $(x, y) \in N_1(A)$ to denote $1 \otimes x + \delta \otimes y \in (\mc{O}_K \otimes A)^\times_1$.
\end{defn}

\begin{prop}\label{prop:Curve Points are Norm 1 Units}
	Let $L/\Q$ be a number field. Then $\mc{C}(\mc{O}_L) \cong N_1(\mc{O}_L)$ as $\Gal(L/\Q)$-modules.
\end{prop}
\begin{proof}
	We map $(x, y) \in \mc{C}(\mc{O}_L) \mapsto (\frac{x - Dy}{2}, y) \in N_1(\mc{O}_L)$. We have that $\frac{x - Dy}{2} \in \mc{O}_L$ because $(x + Dy)(x - Dy) = 4 + (D - D^2)y^2$ and hence $4 \mid (x + Dy)(x - Dy)$. Letting $|\cdot |_2$ be the $2$-adic norm, we have that $|x + Dy|_2 |x - Dy|_2 \le \frac{1}{4}$ and since $x + Dy = x - Dy + 2Dy$, we have that $|x - Dy|_2 \le \frac{1}{2}$ meaning $x - Dy \in 2\mc{O}_L$. It is a straightforward matter to verify that this preserves group structure, respects $\Gal(L/\Q)$ action, and is an isomorphism. 
\end{proof}

\begin{cor}\label{cor:ShaC equals ShaN1}
	$\Sha(\mc{C}/\Z) \cong \Sha(N_1/\Z)$.
\end{cor}
\begin{proof}
	This immediately follows from $H^1(\Q, \mc{C}(\clos{\Z})) \cong \varprojlim H^1(\Gal(L/\Q), \mc{C}(\mc{O}_L))$ over all finite extensions $L/\Q$.
\end{proof}

\begin{remark}
	We look at $\Sha(N_1/\Z)$ because $N_1$ is an algebraic group whose $A$ points are defined for all $\Z$-algebras $A$. The group structure on $\mc{C}$ is not well defined for arbitrary points. For example, if $A = \Z[\sqrt{2}] \otimes_\Z \Z[\sqrt{2}]$ and $D = -4$, then $(\sqrt{2} \otimes \sqrt{2}, 0), (0, 1) \in \mc{C}(A)$, but
	\[(\sqrt{2} \otimes \sqrt{2}, 0) + (0, 1) = \left(0, \frac{\sqrt{2}}{2} \otimes \sqrt{2}\right) \not \in \mc{C}(A).\]
\end{remark}

\begin{thm}\label{thm:Bijection between Sha and 2Cl+(K)}
	There exists a bijection between $\Sha(\mc{C}/\Z)$ and $(Cl^+(K))^2$.
\end{thm}
\begin{proof}
	Since $H^1(\Q, \mc{C}(\clos{\Z}))$ is the direct limit over all finite Galois extensions $L/\Q$, it suffices to show the theorem for \textbf{}sufficiently large $L$. First, take a $1$-cocycle $c \in Z^1(L/\Q, \mc{C}(\mc{O}_L))$ that is a coboundary locally. The following diagram commutes.
	\[\begin{tikzcd} H^1(L/\Q, \mc{C}(\mc{O}_L)) \arrow[r] \arrow[d] & \dprod_p H^1(L_v/\Q_p, \mc{C}(\mc{O}_{L_v})) \arrow[d] \\ H^1(L/\Q, \mc{C}(L)) \arrow[r] & \dprod_p H^1(L_v/\Q_p, \mc{C}(L_v))\end{tikzcd}\]
	Because $c$ is in the kernel of the top morphism, the image of $c$ in $H^1(L/\Q, \mc{C}(L))$ is in the kernel of the bottom morphism. Proposition \ref{prop:Cohomology Q points trivial} implies that $c$ is a coboundary and hence $c(\sigma) = \sigma(m) - m$ for some $m \in \mc{C}(L)$. By Proposition \ref{prop:Curve Points are Norm 1 Units}, we can view $c(\sigma)$ as lying in $(\mc{O}_K \otimes \mc{O}_L)^\times_1$, the units of norm $1$. Define a new action of $\Gal(L/\Q)$ on $\mc{O}_K \otimes \mc{O}_L$ by $c_\sigma(x) = c(\sigma)\sigma(x)$. This is a group action because
	\[c_{\sigma \tau}(x) = c(\sigma \tau)\sigma(\tau(x)) = \sigma(c(\tau))c(\sigma)\sigma(\tau(x)) = c(\sigma)\sigma(c(\tau)\tau(x)) = c_\sigma(c_\tau(x)).\]
	In addition, there exists some $m \in (K \otimes \mc{O}_L)^\times_1$ such that $c(\sigma) = \frac{\sigma(m)}{m}$. Fix such an $m$. The fixed points of this action are $x \in \mc{O}_K \otimes \mc{O}_L$ such that $\sigma(mx) = mx$, which means $mx \in K \otimes \Z$. Thus, to each cocycle $c$, we can associate a submodule $\mf{a} \subset K$ defined as
	\[\mf{a} := \{x \in K: m^{-1}(x \otimes 1) \in \mc{O}_K \otimes \mc{O}_L\}.\]
	Since $m^{-1}$ is a finite sum of simple tensors, there exists some integer $n$ such that $nm^{-1} \in \mc{O}_K \otimes \mc{O}_L$. So $\mf{a}$ is a $\mc{O}_K$-submodule of $K$ such that $n\mf{a} \subset \mc{O}_K$ is an ideal. Therefore $\mf{a}$ is a fractional ideal of $K$. Note that our choice of $m$ is unique up to a unit $u \in (K \otimes \Z)^\times_1$, where either $u$ or $-u$ is totally positive since $[K:\Q] = 2$. Hence different choices of $m$ give an ideal up to multiplication by a totally real element, and so the same ideal class in $Cl^+(K)$. In addition, this map factors through $H^1(L/\Q, \mc{C}(\mc{O}_L))$. To see this, for any two cohomologous cocycles $[c] = [c']$, there exists some $u \in (\mc{O}_K \otimes \mc{O}_L)_1^\times$ such that $\sigma(u)c(\sigma) = c'(\sigma)u$ for all $\sigma$. Letting $c(\sigma) = \frac{\sigma(m)}{m}$ and $c'(\sigma) = \frac{\sigma(m')}{m'}$, we get that $um/m' \in (K \otimes \Z)_1^\times$ so there exists $v \in K_1^\times$ such that $um = m'v$. Let the ideal derived from $m$ be $\mf{a}$ and similarly from $m'$ be $\mf{a}'$. Then
	\[\mf{a}' = \{x \in K: m'^{-1}(x \otimes 1) \in \mc{O}_K \otimes \mc{O}_L\} = \{x \in K: vu^{-1}m^{-1}(x \otimes 1) \in \mc{O}_K \otimes \mc{O}_L\}\]
	\[= \{x \in K: vm^{-1}(x \otimes 1) \in \mc{O}_K \otimes \mc{O}_L\} = v^{-1}\mf{a},\]
	where we used the fact that $u \in (\mc{O}_K \otimes \mc{O}_L)_1^\times$. Hence $[\mf{a}'] = [\mf{a}] \in Cl^+(K)$. Therefore we get a map from $H^1(L/K, \mc{C}(\mc{O}_L))$ to $Cl^+(K)$ and restricting it to the cohomology classes that are locally trivial gives a map from $\Sha(\mc{C}/\Z)$ to $Cl^+(K)$.
	
	We need to show that this ideal class $\mf{a}$ belongs to $(Cl^+(K))^2$. We will prove that $N(\mf{a}) = 1$. Lemma \ref{lem:group scheme local kernel} applied to our situation tells us that locally there is an element $m_p \in (\mc{O}_K \otimes \mc{O}_L \otimes \Z_p)^\times_1$ such that $c(\sigma) \otimes 1 = \frac{\sigma(m_p)}{m_p}$ for all $\sigma \in \Gal(L/\Q)$, where $\sigma$ only acts $\mc{O}_L$. It follows that $\sigma((m\otimes 1)/m_p) = (m\otimes 1)/m_p$ for all $\sigma \in \Gal(L/\Q)$ and hence $u_p := (m\otimes 1)/m_p \in (K \otimes \Z \otimes \Z_p)^\times_1$. Note that we can identify $\mf{a}$ with $(K \otimes \Z) \cap m(\mc{O}_K \otimes \mc{O}_L)$, where the intersection is taken in $K \otimes \mc{O}_L$. Since $\Z_p$ is flat over $\Z$ and thus tensoring will preserve pullbacks,
	\begin{align*}
	\mf{a} \otimes \Z \otimes \Z_p &= [(K \otimes \Z) \cap m(\mc{O}_K \otimes \mc{O}_L)] \otimes \Z_p \\
	&= (K \otimes \Z \otimes \Z_p) \cap (m \otimes 1)(\mc{O}_K \otimes \mc{O}_L \otimes \Z_p)\\
	&= (K \otimes \Z \otimes \Z_p) \cap u_p(\mc{O}_K \otimes \mc{O}_L \otimes \Z_p) \\
	&= u_p(\mc{O}_K \otimes \Z \otimes \Z_p).
	\end{align*}
	Therefore we have that
	\[(\mf{a} \otimes \mc{O}_L) \otimes \Z_p = u_p(\mc{O}_K \otimes \mc{O}_L \otimes \Z_p) = (m(\mc{O}_K \otimes \mc{O}_L)) \otimes \Z_p.\]
	This identity holds independent of our choices because our choice of $m_p$ is up to a unit in $(\mc{O}_K \otimes \Z \otimes \Z_p)^\times_1$, so our choice of $u_p$ is unique up to a unit in $\mc{O}_K \otimes \Z \otimes \Z_p$. For any $\Z$-module $M$, the completion is $\hat{M}_p = M \otimes \Z_p$. Since $\mf{a} \otimes \mc{O}_L \otimes \Z_p = m(\mc{O}_K \otimes \mc{O}_L \otimes \Z_p)$ for all primes $p$, Lemma \ref{lem:Locally equal implies globally equal} gives us that $\mf{a}\otimes \mc{O}_L = m(\mc{O}_K \otimes \mc{O}_L)$. Let $\alpha, \beta$ denote a basis of $\mf{a}$ and let $m = \alpha \otimes x + \beta \otimes y$. By Lemma \ref{lem:quadratic form isomorphism}, we know that $q_\mf{a}(x, y) = N(m)/N(\mf{a}) \in \mc{O}_L^\times$ and since $N(m) = 1$, we have that $N(\mf{a}) \in \mc{O}_L^\times$. The norm is a positive rational number, and hence $N(\mf{a}) = 1$. Our field $K$ is quadratic and so norm $1$ ideals are of the form $\prod_p \mf{p}^{n_p}_1 \mf{p}^{-n_p}_2$, where the product is taken over all split primes. These ideals are in the same narrow ideal class as $\prod_p \mf{p}^{2n_p}$, which is a squared ideal. So $\mf{a}$ lies in $(Cl^+(K))^2$ and hence we have constructed a map from $\Sha(\mc{C}/\Z)$ to $(Cl^+(K))^2$.
	
	By the reasoning from before, in any narrow ideal class $[\mf{a}]$ in $(Cl^+(K))^2$, there exists a fractional ideal $\mf{a}$ of norm $1$. Choose such a fractional ideal. By Corollary $\ref{cor:binary quadratic form 1}$, there exist $x, y \in \mc{O}_L$ for some finite extension $L$ such that $q_\mf{a}(x, y) = 1$. Letting $m = \alpha \otimes x + \beta \otimes y$, Lemma \ref{lem:quadratic form isomorphism} tells us that $\mf{a}\otimes \mc{O}_L = m(\mc{O}_K \otimes \mc{O}_L)$ and so $N(m) = q_\mf{a}(x, y) \cdot N(\mf{a}) = 1$, meaning $m \in (K \otimes \mc{O}_L)^\times_1$. Associate $\mf{a}$ with the cocycle $c(\sigma) = \frac{\sigma(m)}{m}$. We have that
	\[\sigma(m)(\mc{O}_K \otimes \mc{O}_L) = \sigma(\mf{a} \otimes \mc{O}_L) = \mf{a} \otimes \mc{O}_L = m(\mc{O}_K \otimes \mc{O}_L),\]
	and combining this with $N(m) = 1$ gives $\frac{\sigma(m)}{m} \in (\mc{O}_K \otimes \mc{O}_L)^\times_1$. So $c$ does lie in $H^1(L/\Q, \mc{C}(\mc{O}_L))$.
	
	Next we show that $c$ is locally a coboundary. Since $N(\mf{a}) = 1$, if $p$ is inert or ramified, we know that $\mf{a} \otimes \Z_p = \mc{O}_K \otimes \Z_p$. In that case, let $u_p = 1$. If $(p) = \mf{p}_1 \mf{p}_2$ is split, then since $N(\mf{a}) = 1$, we know that $\mf{a} \otimes \Z_p = \mf{p}_1^{n_p} \times \mf{p}_2^{-n_p}$. Let $\pi$ be a uniformizer of $\Z_p$ so we get $\mf{a} \otimes \Z_p = (\pi^{n_p}, \pi^{-n_p}) \mc{O}_K \otimes \Z_p$. In this case, let $u_p = (\pi^{n_p}, \pi^{-n_p})$. Thus for every $p$, we have found $u_p \in (K \otimes \Z_p)^\times_1$ such that $\mf{a} \otimes \Z_p = u_p(\mc{O}_K \otimes \Z_p)$. For every $p$, we have
	\[u_p(\mc{O}_K \otimes \mc{O}_L \otimes \Z_p) = \mf{a}\otimes \mc{O}_L \otimes \Z_p = m(\mc{O}_K \otimes \mc{O}_L \otimes \Z_p)\]
	and hence $m = m_pu_p$ for some unit $m_p \in (\mc{O}_K \otimes \mc{O}_L \otimes \Z_p)^\times_1$. Since $u_p \in K \otimes \Z_p$, we know that $\frac{\sigma(m)}{m} = \frac{\sigma(m_p)}{m_p} \in (\mc{O}_K \otimes \mc{O}_L \otimes \Z_p)_1^\times$ and taking the projection to $\mc{O}_{L_v}$ gives us that $c$ is a coboundary locally. Finally, our choice of $\mf{a}$ is unique up to multiplication by a principal ideal generated by a totally real element. We chose $\mf{a}$ such that $N(\mf{a}) = 1$ so our choice of ideal is unique up to multiplication by a $u \in K^\times_1$, which does not affect our cohomology class. In addition, the choice of $m$ is unique up to multiplication by a $u \in (\mc{O}_K \otimes \mc{O}_L)^\times_1$, which gives a cohomologous class. Accordingly, this map is a well defined map from $(Cl^+(K))^2$ to $\Sha(\mc{C}/\Z)$.
	
	To complete the bijection, we show that these two maps are inverses of each other. Given a cocycle $c \in \Sha(\mc{C}/\Z)$, let $c(\sigma) = \frac{\sigma(m)}{m}$. From this we get an ideal $\mf{a}$ such that $\mf{a} \otimes \mc{O}_L = m(\mc{O}_K \otimes \mc{O}_L)$. Then from this ideal, since we showed that we can make an arbitrary choice of $m'$ such that $\mf{a} \otimes \mc{O}_L = m'(\mc{O}_K \otimes \mc{O}_L)$, we choose this particular $m' = m$ and get the same cocycle and cohomology class back. For the reverse direction, given such an ideal $\mf{a}$ with $N(\mf{a}) = 1$, let $\mf{a} \otimes \mc{O}_L = m(\mc{O}_K \otimes \mc{O}_L)$, giving us the cocycle $c(\sigma) = \frac{\sigma(m)}{m}$. Going back gives the ideal
	\[(K \otimes \Z) \cap m(\mc{O}_K \otimes \mc{O}_L) = (K \otimes \Z) \cap (\mf{a} \otimes \mc{O}_L) = \mf{a},\]
	completing the proof of the bijection.
\end{proof}

\begin{remark}
	By using the same strategy as above, one can also prove the classical result that
	\[\#\ker\left[H^1(K, \mc{O}_{\clos{K}}^\times) \to \prod_v H^1(K_v, \mc{O}_{\clos{K_v}}^\times)\right] = \#Cl(K).\]
	To do so, note that $\mc{O}_{\clos{K}}^\times$ is just the $\mc{O}_{\clos{K}}$ points of $\mb{G}_m$ and then using the same proof strategy as above, show that there exists a bijection between $\Sha(\mb{G}_m/\Z)$ and $Cl(K)$. The analog to Lemma \ref{lem:quadratic form isomorphism} is the fact that for large enough extensions, all ideals of $K$ become principal.
\end{remark}

\section{Descent Theory}
Now that we have proven Theorem \ref{thm:Bijection between Sha and 2Cl+(K)}, we aim to provide a greater geometrical interpretation for what $\Sha(\mc{C}/\Z)$ is classifying. In order to do so, we must first introduce the notion of faithfully flat descent.

\begin{defn}
	Let $R \to S$ be a faithfully flat ring extension. If $M$ is an $S$-module, then $M \otimes_R S$ is an $S \otimes_RS$-module in two ways: For $a \otimes b \in M \otimes_RS$ and $s \otimes t \in S \otimes_RS$, first by $(s \otimes t)(a \otimes b) = as \otimes bt$ and by $(s \otimes t)(a \otimes b) = at \otimes bs$. For a map $\theta: M \otimes_R S \to M \otimes_R S$, create three twistings on $M \otimes S \otimes S$, denoted by $\theta^0, \theta^1, \theta^2$. If $\theta(m \otimes a) = \sum m_i \otimes a_i$, then
	\begin{align*}
	\theta^0(m \otimes u \otimes a) &= \sum m_i \otimes u \otimes a_i,\\
	\theta^1(m \otimes u \otimes a) &= \sum m_i \otimes a_i \otimes u,\\
	\theta^2(m \otimes a \otimes u) &= \sum m_i \otimes a_i \otimes u.
	\end{align*}
	\textbf{Descent data} on $M$ consist of bijections $\theta: M \otimes S \to M \otimes S$ which are isomorphisms from one $(S \otimes S)$ structure to the other, and satisfy $\theta^1 = \theta^0 \theta^2$.
\end{defn}

\begin{remark}
	If $N$ is an $R$-module and $M = N \otimes S$, then the bijection $\theta: N \otimes S \otimes S \to N \otimes S \otimes S$ by $\theta(n \otimes a \otimes b) = n \otimes b \otimes a$ is descent data on $M \otimes S$. The twistings are $\theta^0(n \otimes a \otimes b \otimes c) = n \otimes c \otimes b \otimes a$, $\theta^1(n \otimes a \otimes b \otimes c) = n \otimes c \otimes a \otimes b$, and $\theta^2(n \otimes a \otimes b \otimes c) = n \otimes b \otimes a \otimes c$.
\end{remark}

\begin{thm}
	$R$-modules are naturally equivalent to $S$-modules with descent data
\end{thm}
\begin{proof}
	Given an $S$-module $M$ with descent data $\theta$, we identify it with the $R$-module $N = \{m \in M: \theta(m \otimes 1) = m \otimes 1\}$ and map an $R$-module $N$ to the $S$-module $N \otimes S$ along with the descent data $\theta(n \otimes a \otimes b) = n \otimes b \otimes a$. The proof that this defines a bijective correspondence is in Waterhouse Section $17.2$ \cite{waterhouse}.
\end{proof}

\begin{remark}
	If $M$ has a bilinear map and descent data preserving this map, then this bilinear map comes from the associated $R$-module.
\end{remark}

\begin{defn}
	Suppose $N$ is an $R$-module with some algebraic structure. An \textbf{$S/R$ form of $N$}, or \textit{a twisted form split by $S$}, is another $R$-module with the same type of structure that becomes isomorphic when tensored with $S$.
\end{defn}

\begin{defn}
	Let $N$ be an $R$-module and $\varphi$ an automorphism of $N \otimes S \otimes S$. If $\varphi(n \otimes a \otimes b) = \sum n_i \otimes a_i \otimes b_i$, then three twistings of $\varphi$ are $d^0\varphi, d^1\varphi, d^2\varphi\in Aut(N \otimes S \otimes S \otimes S)$ by
	\begin{align*}
	d^0\varphi(n \otimes u \otimes a \otimes b) &= \sum n_i \otimes u \otimes a_i \otimes b_i,\\
	d^1\varphi(n \otimes a \otimes u \otimes b) &= \sum n_i \otimes a_i \otimes u \otimes b_i,\\
	d^2\varphi(n \otimes a \otimes b \otimes u) &= \sum n_i \otimes a_i \otimes b_i \otimes u.
	\end{align*}
	For an $S$ automorphism $\lambda$ of $N \otimes S$, suppose that $\lambda(n \otimes a) = \sum n_i \otimes a_i$. Define $d^0\lambda, d^1\lambda \in Aut(N \otimes S \otimes S)$ in a similar manner.
\end{defn}

\begin{cor}\label{cor:Descent Theory}
	The $S/R$ forms of $N$ are naturally equivalent to automorphisms $\varphi$ of $N \otimes S \otimes S$ preserving the algebraic structure and satisfying $(d^0\varphi)(d^2\varphi) = d^1\varphi$, modulo the equivalence relation $\varphi \sim \varphi'$ if there exists some $\lambda \in Aut(N \otimes S)$ such that $\varphi' = (d^0\lambda)\varphi(d^1\lambda)^{-1}$.
\end{cor}
\begin{proof}
	Given an automorphism $\varphi$, we associate it with descent data by taking $\theta \circ \varphi$, where $\theta(n \otimes a \otimes b) = n \otimes b \otimes a$. The details are shown in Waterhouse Section $17.5$ \cite{waterhouse}.
\end{proof}

\section{Application of Descent Theory}
\begin{prop}
	Let $L/\Q$ be a finite Galois extension. The map $\mc{O}_L \otimes_\Z \mc{O}_L \to \dprod_{\sigma \in \Gal(L/\Q)} \mc{O}_L$ is injective.
\end{prop}
\begin{proof}
	By the primitive element theorem and the fact that $L/\Q$ is Galois, we have that $L \otimes_{\Q} L \cong \prod L$ through the map $a \otimes b \mapsto (\sigma(a)b)_\sigma$. Furthermore, we have the commutative square
	\[\begin{tikzcd} \mc{O}_{L} \otimes_{\Z} \mc{O}_{L} \arrow[r] \arrow[d] & \dprod \mc{O}_{L} \arrow[d]\\
	L \otimes_{\Q} L \arrow[r] & \dprod L\end{tikzcd}\]
	The bottom morphism is an isomorphism and the left morphism is injective because $\mc{O}_{L}$ is free over $\Z$ and a basis for $\mc{O}_{L}$ over $\Z$ is a basis for $L$ over $\Q$. Thus, the top morphism must be injective.
\end{proof}

\begin{remark}
	Each element in $Z^1(L/\Q, N_1(\mc{O}_L))$ is a function and can be viewed as lying in $\dprod_{\Gal(L/\Q)} N_1(\mc{O}_L) = N_1\left(\dprod_{\Gal(L/\Q)} \mc{O}_L\right)$, where the $\sigma$ component of the product is the value of the function at $\sigma$.
\end{remark}
\begin{lem}\label{lem:coboundaries lie in descent data}
	For every prime $p$ and place $v$ of $L$ lying above $p$, we can view $B^1(L/\Q, N_1(\mc{O}_L \otimes \Z_p)) \subset N_1(\mc{O}_L \otimes \mc{O}_L \otimes \Z_p)$.
\end{lem}
\begin{proof}
	Let $f \in B^1(L/\Q, N_1(\mc{O}_L \otimes \Z_p))$, then there exists $(x, y) \in N_1(\mc{O}_L \otimes \Z_p)$ such that $f = (\sigma(x), \sigma(y))_\sigma - (x, y)_\sigma = (1 \otimes \sigma(x) + \frac{D+ \sqrt{D}}{2} \otimes \sigma(y))_\sigma \cdot (1 \otimes x + \frac{D - \sqrt{D}}{2} \otimes y)_\sigma$. Letting $\delta = \frac{D+ \sqrt{D}}{2}$ and using the fact that $\delta^2 = D\delta + \frac{D - D^2}{4}$, we have that
	\[f = 1 \otimes \left(\sigma(x)(x + Dy) - \frac{D - D^2}{4}\sigma(y)y\right)_\sigma + \delta \otimes \left(\sigma(y)(x - Dy) - \sigma(x)y\right)_\sigma.\]
	Let $\iota_0: \mc{O}_L \otimes \Z_p \to \mc{O}_L \otimes \mc{O}_L \otimes \Z_p$ be the map $a \otimes b \mapsto a \otimes 1 \otimes b$ and $\iota_1$ be the map $a \otimes b \mapsto 1 \otimes a \otimes b$. Then $f$ is
	\[\left(\iota_0(x) \cdot \iota_1(x + Dy) - \iota_0\left(\frac{D - D^2}{4} y\right) \cdot \iota_1(y), \iota_0(y) \cdot \iota_1(x - Dy) - \iota_0(x) \cdot \iota_1(y)\right),\]
	which is an element of $N_1(\mc{O}_{L} \otimes \mc{O}_L \otimes \Z_p)$.
\end{proof}

\begin{prop}\label{prop:Kernel is Automorphism}
	If
	\[f \in \ker\left(Z^1(L/\Q, N_1(\mc{O}_L)) \to \prod_p H^1(L_v/\Q_p, N_1(\mc{O}_{L_v}))\right),\]
	then in fact $f \in N_1(\mc{O}_L \otimes_\Z \mc{O}_L)$.
\end{prop}
\begin{proof}
	Lemma \ref{lem:group scheme local kernel} tells us that $f$ lies in
	\[\bigcap_p \ker \left(Z^1(L/\Q, N_1(\mc{O}_L)) \to H^1(L/\Q, N_1(\mc{O}_L \otimes \Z_p))\right).\]
	By viewing the element $f$ as lying in $N_1(\dprod \mc{O}_L)$, Lemma \ref{lem:coboundaries lie in descent data} says it suffices to show that if an element $x \in \dprod \mc{O}_L$ is such that for all primes $p$, when viewing $x$ as an element of $\dprod \mc{O}_L \otimes \Z_p$, it lies in the image of $\mc{O}_L \otimes \mc{O}_L \otimes \Z_p$ inside $\dprod \mc{O}_L \otimes \Z_p$, then $x$ lies in the image of $\mc{O}_L \otimes_\Z \mc{O}_L$. More precisely, for all primes $p$, we have the commutative diagram of short exact sequences
	\[\begin{tikzcd} 0 \arrow[r] &\mc{O}_L \otimes_{\Z} \mc{O}_L \arrow[r] \arrow[d] & \dprod_{\Gal(L/\Q)} \mc{O}_L \arrow[r] \arrow[d] & \mf{a} \arrow[r] \arrow[d] & 0\\
	0 \arrow[r] &\mc{O}_L \otimes \mc{O}_L \otimes \Z_p \arrow[r] & \dprod_{\Gal(L/\Q)} \mc{O}_L \otimes \Z_p \arrow[r] & \mf{a} \otimes \Z_p \arrow[r] & 0\end{tikzcd}\]
	where $\mf{a}$ denotes the cokernel of the map, and the bottom sequence is exact because $\Z_p$ is flat over $\Z$. We are given an element $x \in \dprod \mc{O}_L$ whose image in $\mf{a}$ becomes trivial in $\mf{a} \otimes \Z_p$. Moreover, both $\mc{O}_L \otimes_Z \mc{O}_L$ and $\dprod \mc{O}_L$ have the same finite rank over $\Z$ and so the cokernel $\mf{a}$ must be finite. Letting $r$ be the order of $x$ in $\mf{a}$, the condition that $x$ is trivial in $\mf{a} \otimes \Z_p$ implies that $(r, p) = 1$ for all primes $p$. Consequently, the order of $x$ must be $1$ and hence $x$ lies in the image of $\mc{O}_L \otimes \mc{O}_L$.
\end{proof}

\begin{thm}\label{thm:quadratic form narrow class group bijection}
	There is a bijection between equivalence classes of primitive quadratic forms of discriminant $D$ modulo congruence by elements of $SL_2(\Z)$ and pairs $(\mf{a}, s)$ with $\mf{a} \subset K$ a fractional ideal and $s = \pm 1$ modulo the equivalence relation $(\mf{a}, s) \sim (\mf{b}, s')$ if and only if there exists some $x \in K$ such that $\mf{a} = x\mf{b}$ and $s = \frac{N(x)}{|N(x)|}s'$.
\end{thm}
\begin{proof}
	Choosing a basis $\alpha, \beta$ of $\mf{a}$, we map $(\mf{a}, s)$ to the binary quadratic form
	\[q_\mf{a}(x, y) = s \cdot \frac{N(\alpha x + \beta y)}{N(\mf{a})}.\]
	Refer to Cohen \cite[Thm.~5.2.9]{cohen} for details showing injectivity and surjectivity.
\end{proof}

\begin{lem}
	Let $\mf{a} \subset K$ be a fractional ideal and $f_\mf{a}: \mf{a}^2 \to \Z$ be an anti-symmetric bilinear form. Let $\alpha, \beta$ be a basis of $\mf{a}$ chosen such that $\frac{\bar{\alpha}\beta - \alpha \bar{\beta}}{\sqrt{D}} = N(\mf{a}) > 0$. Let $A$ be a ring. Then we can extend $f_\mf{a}$ to $f_\mf{a}: (\mf{a} \otimes A)^2 \to A$ and the norm map $N: K \to \Z$ to $N: K \otimes A \to A$. Moreover, for $m \in \mf{a} \otimes A$,
	\[f_\mf{a}(m, m\delta) = \frac{N(m)}{N(\mf{a})} f_\mf{a}(\alpha, \beta).\]
\end{lem}
\begin{proof}
	In general, we have that
	\[f_\mf{a}(\alpha \otimes x_1 + \beta \otimes y_1, \alpha \otimes x_2 + \beta \otimes y_2) = (x_1y_2 - x_2y_1)f_\mf{a}(\alpha, \beta) = \frac{z_1 \overline{z_2} - z_2 \overline{z_1}}{\alpha \overline{\beta} - \beta \overline{\alpha}} f_\mf{a}(\alpha, \beta),\]
	where $z_1 = \alpha \otimes x_1 + \beta \otimes y_1$ and $z_2 = \alpha \otimes x_2 + \beta \otimes y_2$, and $\overline{\alpha \otimes x + \beta \otimes y} = \overline{\alpha} \otimes x + \overline{\beta} \otimes y$. Combining this with $\overline{\delta} - \delta = -\sqrt{D}$ and $N(m) = m \cdot \overline{m}$, we get our result.
\end{proof}

\begin{thm}
	$\Sha(\mc{C}/\Z)$ has an interpretation as equivalence classes of binary quadratic forms with discriminant $D$ that represent $1$ in $\Z_p$ for all primes $p$.
\end{thm}
\begin{proof}
	It suffices to look at $\Sha(N_1/\Z)$, which is the projective limit over all finite extensions $L/\Q$. So, for fixed $L/\Q$, consider
	\[\ker\left(H^1(L/\Q, N_1(\mc{O}_L)) \to \prod_p H^1(L_v/\Q_p, N_1(\mc{O}_{L_v}))\right).\]
	By Proposition \ref{prop:Kernel is Automorphism}, this is precisely elements of $N_1(\mc{O}_L \otimes \mc{O}_L)$, up to a coboundary condition, that satisfy the $1$st order cochain condition and become coboundaries locally.
	
	We use descent theory. We know that $\mc{O}_L$ is free, and hence faithfully flat, over $\Z$. Thus, Corollary \ref{cor:Descent Theory} tells us that $\mc{O}_L/\Z$ forms of $\mc{O}_K$ are naturally equivalent to automorphisms of $\mc{O}_K \otimes \mc{O}_L \otimes \mc{O}_L$ satisfying a relation between $d^0\varphi, d^1\varphi, d^2\varphi$, modulo a relation in terms of $d^0\lambda, d^1\lambda$. The additional algebraic structure we impose is a structure of $\mc{O}_K$ as an $\mc{O}_K$-module, and an antisymmetric bilinear map $f: \mc{O}_K^2 \to \Z$. The $2$-form we endow $\mc{O}_K$ with is determined by $f(1, \delta) = 1$. All $\mc{O}_K \otimes \mc{O}_L \otimes \mc{O}_L$-linear automorphisms of itself are given by units $a \in (\mc{O}_K \otimes \mc{O}_L \otimes \mc{O}_L)^\times$. The fact that this automorphism must preserve this map means that $f(a, a\delta) = 1$. A quick calculation shows $f(a, a \delta) = N(a) = 1$ and so we are looking for automorphisms that are $\mc{O}_L \otimes \mc{O}_L$ points of $N_1$. Given $\varphi = (\sum a_i \otimes b_i, \sum c_i \otimes d_i) \in N_1(\mc{O}_L \otimes \mc{O}_L)$, the condition that $(d^0\varphi)(d^2\varphi) = d^1\varphi$ tells us that
	\[\left(\sum 1 \otimes a_i \otimes b_i, \sum 1 \otimes c_i \otimes d_i\right) + \left(\sum a_i \otimes b_i \otimes 1, \sum c_i \otimes d_i \otimes 1\right)\]
	\[= \left(\sum a_i \otimes 1 \otimes b_i, \sum c_i \otimes 1 \otimes d_i\right).\]
	We can embed $\mc{O}_L \otimes \mc{O}_L \otimes \mc{O}_L \into \prod_{\Gal(L/\Q)}\prod_{\Gal(L/\Q)} \mc{O}_L$ via $a \otimes b \otimes c \mapsto (\sigma(a)\tau(b)c)_{\sigma, \tau}$ and hence this condition is equivalent to
	\[\left(\sum \tau(a_i)b_i, \sum \tau(c_i)d_i\right) + \left(\sum \sigma(a_i)\tau(b_i), \sum \sigma(c_i)\tau(d_i)\right)\]
	\[= \left(\sum \sigma(a_i)b_i, \sum \sigma(c_i)d_i\right).\]
	However, notice that this is exactly the same as the condition for $1$ cochains which says that $f(\sigma) = f(\tau) + \tau f(\tau^{-1} \sigma)$ for all $\sigma, \tau$. Finally, we say that two automorphisms $\varphi, \varphi' \in N_1(\mc{O}_L \otimes \mc{O}_L)$ give isomorphic modules if there exists some $\lambda \in N_1(\mc{O}_L)$ such that $\varphi' = d^0\lambda + \varphi - d^1\lambda$. Letting $\lambda = (x, y) \in N_1(\mc{O}_L)$, we have that $\varphi - \varphi' = (x \otimes 1, y \otimes 1) - (1 \otimes x, 1 \otimes y) = (\sigma(x), \sigma(y))_\sigma - (x, y)_\sigma \in N_1(\prod \mc{O}_L)$, which is precisely the $1$ coboundary condition.
	
	The condition that locally the descent data is of the form $d^0\lambda - d^1\lambda$ for some $\lambda \in N_1(\mc{O}_L \otimes \Z_p)$ says that this $\mc{O}_L/\Z$ form of $\mc{O}_K$, when viewed as a $(\mc{O}_L \otimes \Z_p)/\Z_p$ form of $\mc{O}_K \otimes \Z_p$, is already isomorphic to $\mc{O}_K \otimes \Z_p$. Therefore, we can characterize $\Sha(N_1/\Z)$ as $\mc{O}_K$-modules $\mf{a}$ with an anti-symmetric bilinear map $f_\mf{a}: \mf{a}^2 \to \Z$ satisfying the following conditions:
	\begin{enumerate}[(1)]
		\item There is an isomorphism $\mf{a} \otimes \mc{O}_L \cong \mc{O}_K \otimes \mc{O}_L$ of $\mc{O}_K \otimes \mc{O}_L$-modules that preserves the map. And
		
		\item for every prime $p$, there is an isomorphism $\mf{a} \otimes \Z_p \cong \mc{O}_K \otimes \Z_p$ of $\mc{O}_K \otimes \Z_p$-modules preserving the bilinear map as well.
	\end{enumerate}
	Two modules are equivalent $(\mf{a}, f_\mf{a}) \cong (\mf{b}, f_\mf{b})$ if there is an $\mc{O}_K$-linear isomorphism $\mf{a} \to \mf{b}$ that sends $f_\mf{a}$ to $f_\mf{b}$.
	
	By the decomposition theorem for modules over Dedekind domains and the fact that $\mf{a} \otimes \mc{O}_L \cong \mc{O}_K \otimes \mc{O}_L$, we know that $\mf{a}$ must be torsion free and isomorphic to an ideal of $\mc{O}_K$. Consequently, $\mf{a}$ can be viewed as a fractional ideal of $K$, so it is a free $\Z$-module of rank $2$. Choose two generators $\alpha = ca, \beta = c\frac{b + \sqrt{D}}{2}$ and note that $\frac{\bar{\alpha}\beta - \alpha \bar{\beta}}{\sqrt{D}} > 0$.
	
	Lemma \ref{lem:quadratic form isomorphism} tells us that there is an isomorphism $\mc{O}_K \otimes \mc{O}_L\xrightarrow{\sim} \mf{a} \otimes \mc{O}_L$ if and only if there exist $x, y \in \mc{O}_L$ such that $q_\mf{a}(x, y)  \in \mc{O}_L^\times$, and this isomorphism sends $1 \otimes 1 \mapsto \alpha \otimes x + \beta \otimes y$. The bilinear map must be preserved and it suffices to require
	\[1 = f(1, \delta) = f_\mf{a}(\alpha \otimes x + \beta \otimes y, \alpha \delta \otimes x + \beta \delta \otimes y) = q_\mf{a}(x, y)f_\mf{a}(\alpha, \beta).\]
	Consequently, the isomorphism preserves the bilinear map as well if and only if $f_\mf{a}(\alpha, \beta) = \pm 1$ and $q_\mf{a}$ represents $f_\mf{a}(\alpha, \beta)$ in $\mc{O}_L$.
	
	Returning to the problem at hand, we wish to classify pairs $(\mf{a}, s)$, where $\mf{a}$ is a fractional ideal of $K$ and $s = \pm1 = f_\mf{a}(\alpha, \beta)$ which satisfy certain conditions, modulo an equivalence relationship. This relation is $(\mf{a}, s) \sim (\mf{b}, s')$ if and only if there exists some $x \in K^\times$ such that $\mf{a} = x\mf{b}$ and preserves the bilinear map, which corresponds to $s = \frac{N(x)}{|N(x)|} s'$. The first property that a pair $(\mf{a}, s)$ must satisfy is that there is a bilinear form preserving isomorphism $\mf{a} \otimes \mc{O}_L \cong \mc{O}_K \otimes \mc{O}_L$, which is equivalent to $s$ being representable by $q_\mf{a}$ in $\mc{O}_L$. By Corollary \ref{cor:binary quadratic form 1}, this always holds for all large enough $L$. The second condition is that for every prime $p$, there is a bilinear form preserving isomorphism $\mf{a} \otimes \Z_p \cong \mc{O}_K \otimes \Z_p$, which is equivalent to $s$ being representable by $q_\mf{a}$ in $\Z_p$. Thus $\Sha(\mc{C}/\Z)$ corresponds to equivalence classes of binary quadratic forms $s\mf{q}_a$ that represent $1$ in $\Z_p$ for all primes $p$.
\end{proof}

\begin{remark}
	The proof in Theorem \ref{thm:Bijection between Sha and 2Cl+(K)} uses the same idea of descent data implicitly. To see this, giving an automorphism of $\mc{O}_K \otimes \mc{O}_L \otimes \mc{O}_L$ that preserves the bilinear form is equivalent to giving an element $m \in (\mc{O}_K \otimes \mc{O}_L \otimes \mc{O}_L)_1^\times$. Write $m_\sigma \in \mc{O}_K \otimes \mc{O}_L$ to be the $\sigma \in \Gal(L/\Q)$ projection of $m$ when $m$ is viewed as an element in $\dprod \mc{O}_K \otimes \mc{O}_L$. Then, the $\mc{O}_K$-module associated with $m$ is the set
	\[\{x \in \mc{O}_K \otimes \mc{O}_L: m \cdot (x \otimes 1) = \theta(x \otimes 1)\} = \{x \in \mc{O}_K \otimes \mc{O}_L: \forall \sigma \in \Gal(L/\Q), m_\sigma \cdot \sigma(x) = x\}.\]
	This is exactly the fixed points of $\mc{O}_K \otimes \mc{O}_L$ under the twisted action by a cocycle $c$ since we viewed a cocycle $c$ as an element of $\dprod \mc{O}_K \otimes \mc{O}_L$.
\end{remark}

\begin{remark}
	This proof requires more machinery than the proof of Theorem \ref{thm:Bijection between Sha and 2Cl+(K)}, but it provides an enlightening interpretation for the Tate--Shafarevich group as measuring the failure of the Hasse principle over the integers. To be more explicit, the Hasse principle states that a binary quadratic form $Q$ represents a rational number $q \in \Q$ if and only if it represents $q$ in $\Q_p$ for all primes $p$. This is justification for why Proposition \ref{prop:Cohomology Q points trivial} holds. With our interpretation, the Tate--Shafarevich group is quantifying how many binary quadratic forms represent $1$ in $\Z_p$ for all primes $p$, but do not represent $1$ in $\Z$. Thus in this sense, the group is capturing the failure of the Hasse principle over the integers.
\end{remark}

	\bibliographystyle{acm}
	\bibliography{thesis}
	
\end{document}